\documentclass[12pt]{amsart}
\usepackage{geometry}                
\usepackage{graphicx}
\usepackage{amsfonts,latexsym,rawfonts,amsmath,amssymb,amsthm}
\usepackage{epstopdf}
\usepackage{cite}
\usepackage[plainpages=false]{hyperref}
\RequirePackage{color}

\newtheorem{theorem}{Theorem}[section]

\newtheorem{corollary}{Corollary}[section]

\newtheorem{lemma}[theorem]{Lemma}

\newtheorem{remark}[theorem]{Remark}

\numberwithin{equation}{section}

\newcommand{\beq}{\begin{equation}}
\newcommand{\eeq}{\end{equation}}
\newcommand{\beqo}{\begin{equation*}}
\newcommand{\eeqo}{\end{equation*}}
\newcommand{\beqs}{\begin{eqnarray*}}
\newcommand{\eeqs}{\end{eqnarray*}} 
\newcommand{\beqn}{\begin{eqnarray}}
\newcommand{\eeqn}{\end{eqnarray}}

\def\eps{\varepsilon}

\def\rho{\varrho}

\begin{document}
\title{Interior $C^2$ estimate for Monge-Amp\`ere equation in dimension two}

\author[J. Liu]
{Jiakun Liu}
\address
	{
	School of Mathematics and Applied Statistics,
	University of Wollongong,
	Wollongong, NSW 2522, AUSTRALIA}
\email{jiakunl@uow.edu.au}



\maketitle

\baselineskip18pt

\parskip10pt

\begin{abstract}
We obtain a genuine local $C^2$ estimate for the Monge-Amp\`ere equation in dimension two, by using the partial Legendre transform. 
\end{abstract}

\section{Introduction}

Consider the Monge-Amp\`ere equation 
	\begin{equation}\label{MA}
		\det\,D^2u = f \quad\mbox{ in } B_R(0)\subset\mathbb{R}^2,
	\end{equation}
where $f$ is a given positive function. 
This equation has been intensively studied since the last century, which is largely motivated by related geometric problems, in particular the Weyl and Minkowski problems. 
We refer the readers to \cite{N53} for more discussions on these two problems and their connection with the Monge-Amp\`ere equation \eqref{MA}.

It is well known that a \emph{genuine} interior $C^2$ estimate is an important and challenging ingredient in obtaining the regularity of weak solutions to elliptic equations. 
By genuine, we mean that the estimate is uniform along a convergent sequence of continuous solutions, which is also sometimes called \emph{purely} interior $C^2$ estimate.
For the Monge-Amp\`ere equation in dimension $n=2$, the corresponding interior $C^2$ estimate is due to Heinz \cite{H59}.  
From the famous counterexample of Pogorelov \cite{P}, there is no genuine interior $C^2$ estimate for the Monge-Amp\`ere equation \eqref{MA} in higher dimensions ($n\geq3$). Instead, Pogorelov \cite{P} obtained an interior $C^2$ estimate for solutions satisfying an affine boundary condition, which is now called Pogorelov-type estimate, see \cite[\S17]{GT}.  

In his proof \cite{H59}, Heinz utilised a ``characteristic" theory that was previously exploited by Lewy for analytic Monge-Amp\`ere equations, and then derived an interior $C^2$ estimate by a property of univalent mappings that he established in \cite{H56}.
Although there is no explicit form of the $C^2$ estimate given in \cite{H59} (see Theorem 3 therein), we believe that by carefully tracing all constants in \cite[Theroem 11]{H56}, one should be able to derive an explicit $C^2$ estimate but that seems quite involved. 
Recently, another interesting proof using only maximum principle is given by Chen-Han-Ou in \cite{CHO} with a concise form that
	\begin{equation}\label{CHOest}
		|D^2u(0)| \leq C_1 e^{C_2 \sup_{B_R(0)}|u|^2/R^4}, 
	\end{equation}
where $C_1, C_2$ are some constants depending only on $\|f\|_{C^{1,1}}$, but independently of $u$. 
The method of \cite{CHO} is based on a suitable choice of auxiliary function involving the second derivatives of $u$ with respect to a continuous eigenvector field, and by an arduous calculation similar to Pogorelov's \cite{P}, the estimate \eqref{CHOest} can then be derived.

In this paper, we give a different and much simpler way of obtaining estimates analogous to \eqref{CHOest}, which also enables to reduce the regularity assumption on $f$. 
Our method is based on the partial Legendre transform, in which the related coordinate transform has been used in Lewy and Heinz's works on Monge-Amp\`ere equations in dimension two and can be traced back to Darboux in the 19th century. We refer the reader to Schulz's book \cite{Sch} for an excellent exposition of the regularity theory for Monge-Amp\`ere equations in dimension two. 
The partial Legendre transform is also useful in studying certain degenerate Monge-Amp\`ere equations, see e.g. \cite{DS,G97}.
For more recent development and applications, we refer the reader to the papers \cite{GP,L17} and references therein.

Now, we state our main result.
\begin{theorem}\label{mainthm}
Let $u\in C^4(B_R(0))$ be a convex solution to \eqref{MA} satisfying $u(0)=0$ and $Du(0)=0$. 
Assume that $0<C_0^{-1}\leq f\leq C_0$ and $f\in C^{\alpha}(B_R(0))$ for some $\alpha\in(0,1)$. Then
	\begin{equation}\label{est3}
		|D^2u(0)| \leq C_1 \left(\sup_{B_R}|u|^2 R^{-4} e^{C_2 \sup_{B_R}|u|^2/R^4} + 1 \right), 
	\end{equation}
where the constant $C_1$ depends on $C_0$ and $\|f\|_{C^\alpha (B_R(0))}$, and $C_2$ depends only on $C_0$. 
\end{theorem}

The idea of our proof is that after the partial Legendre transform, equation \eqref{MA} will become a quasilinear uniformly elliptic equation, in particular Laplace's equation when $f$ is a constant. Then by the interior estimates for quasilinear equations \cite{GT} and transforming back, we can obtain the corresponding interior estimates for equation \eqref{MA}. 
A key point is to estimate the \emph{modulus of convexity} $m=m[u]$ of $u$, which is defined by
	\begin{equation}\label{mconv}
		m(t) = \inf \{u(x) - \ell_z(x) : |x-z|>t\},
	\end{equation}
where $t>0$, $\ell_z$ is the supporting function of $u$ at $z$. Obviously $m$ is a nonnegative function of $t$. 
In dimension two, due to Aleksandrov \cite{A42} and Heinz \cite{H59}, the solution $u$ of \eqref{MA} must be strictly convex, and thus $m$ is a positive function. 
In \S\ref{s21}, we reveal an essential connection between the strict convexity and interior $C^2$ estimate.
In order to derive \eqref{est3}, we need a lower bound estimate of the modulus of convexity, which is obtained in \S\ref{s22} by elaborating the approach developed in \cite[\S3]{TW}. 

Last, we remark that the Monge-Amp\`ere equation \eqref{MA} also occurs in the study of special Lagrangian graphs in dimension two. 
The partial Legendre transform changes between the Monge-Amp\`ere and Laplace's equations corresponds to the change of phase $\Theta$ from $\frac{\pi}{2}$ to $0$, see \cite{Y20} for more discussions. 
The above interesting fact motivates and inspires our work in this paper.

\section{Proof of theorem}

The proof of Theorem \ref{mainthm} is divided into two parts: 
in \S\ref{s21}, we first derive an interior $C^2$ estimate in terms of the modulus of convexity $m$; 
and then in \S\ref{s22}, we give a lower bound estimate of $m$ and complete the proof of \eqref{est3}.

\subsection{Interior $C^2$ estimate} \label{s21}
Let $u$ be a convex solution of \eqref{MA} in $B_R\subset\mathbb{R}^2$. 
For a fixed $y_2\in(-R,R)$, the \emph{partial Legendre transform} $u^*$ of $u$ in the $e_1$-direction is given by
	\begin{equation}\label{pLeng}
		u^*(y_1,y_2) := \sup\{x_1y_1 - u(x_1,y_2)\},
	\end{equation}
where the supremum is taken on the slice $y_2$ is the fixed constant, namely for all $x_1$ such that $(x_1,y_2)\in B_R$. 
Notice that the supremum is attained when 
	\begin{equation}
		y_1=\partial_{x_1}u(x_1, x_2),\quad\mbox{where } x_2=y_2.
	\end{equation}
One can easily verify that the partial Legendre transform of $u^*$ is $u$, that is $(u^*)^*=u$. 	

Due to Aleksandrov \cite{A42} and Heinz \cite{H59} (see also \cite{Fi,Sch,TW}), $u$ is strictly convex.	
Hence we can define an injective mapping $\mathcal{P} : B_R \to \mathbb{R}^2$ by
	\begin{equation}\label{cotrans}
		x \mapsto y=\mathcal{P}(x)=(u_{x_1}, x_2).
	\end{equation}	
This transform $\mathcal{P}$ has historically been used in the characteristic theory of Monge-Amp\`ere equations, see Schulz's book \cite{Sch} for more details, and has also been applied to degenerate Monge-Amp\`ere equations by Guan in \cite{G97}, while the above work seems not invoke the function $u^*$ in \eqref{pLeng}. 
By differentiation, it is easy to check that
	\begin{equation}
		u^*_{y_1} = x_1,\quad u^*_{y_2} = -u_{x_2},
	\end{equation}
and
	\begin{equation}\label{u2trans}
		u^*_{y_1y_1} = \frac{1}{u_{x_1x_1}}, \ \ u^*_{y_1y_2} = -\frac{u_{x_1x_2}}{u_{x_1x_1}}, \ \ u^*_{y_2y_2} = -\frac{f}{u_{x_1x_1}}.
	\end{equation}
Therefore, $u^*$ satisfies the quasilinear elliptic equation
	\begin{equation}\label{quasi}
		u^*_{y_1y_1}\, f(u^*_{y_1}, y_2) + u^*_{y_2y_2} = 0 \quad\mbox{in  } \mathcal{P}(B_R).
	\end{equation}
From the assumption that 	$0<C_0^{-1}\leq f\leq C_0$, equation \eqref{quasi} is uniformly elliptic. 
		 
\begin{lemma} \label{lemball}
There exists a constant $\delta>0$ depending on the modulus of convexity $m$ defined in \eqref{mconv},
such that $B_\delta(0)\subset \mathcal{P}(B_R)$.
\end{lemma}

\begin{proof}
Denote $m_0:=\inf_{\partial B_R}u$, thus $m_0\geq m(R)$. 
Let $\hat x\in\partial B_R$. By the convexity of $u$, we have
	\begin{equation}\label{estconv}
	\begin{split}
		m_0 & \leq u(\hat x) \leq Du(\hat x)\cdot \hat x \\ 
			& \leq |u_1(\hat x)| \cdot |\hat x_1| + |u_2(\hat x)| \cdot |\hat x_2| \\
			& \leq |u_1(\hat x)| \cdot R + |Du|_{\partial B_R} \cdot |\hat x_2|.
	\end{split}
	\end{equation}
One can see that 
\begin{itemize}
\item[$(i)$] if $|\hat x_2| \geq \frac{m_0}{2|Du|_{\partial B_R}}$, then $|\mathcal{P}(\hat x)| \geq |\hat x_2| \geq \frac{m_0}{2|Du|_{\partial B_R}}$;
\item[$(ii)$] if $|\hat x_2| < \frac{m_0}{2|Du|_{\partial B_R}}$, then by \eqref{estconv}, $|\mathcal{P}(\hat x)| \geq |u_1(\hat x)| \geq \frac{m_0}{2R}$.
\end{itemize}

Therefore, by setting
	\begin{equation}\label{del}
		\delta := \min\left\{ \frac{m_0}{2|Du|_{\partial B_R}}, \ \  \frac{m_0}{2R}\right\},
	\end{equation}
we have $B_\delta(0)\subset\mathcal{P}(B_R)$.
\end{proof}

\begin{lemma}
There exists a constant $C_1>0$ depending on $|f|_\alpha$, such that 
	\begin{equation} 
		|D^2u^*(0)|  \leq C_1\frac{\|u^*\|_{L^\infty(B_\delta)}}{\delta^2} \left( \frac{\|u^*\|_{L^\infty(B_\delta)}}{\delta^2} +1 \right). 
	\end{equation}
\end{lemma}

\begin{proof}
The proof is essentially from \cite[Theorem 12.4]{GT} for quasilinear elliptic equations in dimension two, and is standard for Laplace's equation when $f\equiv1$ in \eqref{quasi}.
In order to trace the dependency of constants and for the sake of completeness, we outline some main steps as follows. 
In fact, from \cite[Theorem 12.4]{GT}, there exists some $\hat\alpha>0$ depending on $C_0$ such that in $B_r$, ($r={\delta/2}$)
	\begin{equation}\label{quasialpha}
		[u^*]_{1,\hat\alpha} \leq \frac{C}{\delta^{1+\hat\alpha}}|u^*|_0,
	\end{equation}
for a universal constant $C>0$. Unless otherwise indicated, the universal constant $C$ may vary through the context. 
Recall that $x_1=u^*_{y_1}$ is bounded. From \eqref{quasialpha} we know the function $u^*_{y_1}=u^*_{y_1}(y_1, y_2)$ is $C^{\hat\alpha}$ continuous and for any $y, y'\in B_{\delta/2}$, 
	\begin{equation*}
		\frac{|u^*_{y_1}(y)-u^*_{y_1}(y')|}{|y-y'|^{\hat\alpha}} \leq \frac{C}{\delta^{1+\hat\alpha}}|u^*|_0. 
	\end{equation*}

By the assumption $f\in C^{\alpha}$ and let $\alpha'=\alpha\hat\alpha \in(0,1)$, then one has $f(y)=f(u^*_{y_1}, y_2)$ is $C^{\alpha'}$ continuous in $y$, namely for any $y, y'\in B_{\delta/2}$, 
	\begin{equation}\label{combalpha}
	\begin{split}
		\frac{|f(y)-f(y')|}{|y-y'|^{\alpha'}} &\leq C|f|_{C^\alpha} \big(|u^*_{y_1}|_{\hat\alpha}^\alpha + 1\big) \\
				&\leq C_1\left(\frac{|u^*|_0^\alpha}{\delta^{\alpha(1+\hat\alpha)}} + 1\right)=: C_\Lambda,
	\end{split}
	\end{equation}
where the constant $C_1$ depends on $|f|_\alpha$.  
This implies that the coefficient of the quasilinear elliptic equation \eqref{quasi} is $C^{\alpha'}$ continuous.

Last, by the standard Schauder theory in \cite[Theorem 6.2]{GT}, see also \cite{W06}, that in $B_{\delta/2}$ one has
	$
		|u^*|_{2,\alpha} \leq \frac{C_\Lambda}{\delta^{2+\alpha}}|u^*|_0,
	$
and in particular, 
	\begin{equation}\label{C2at0} 
		|D^2u^*(0)|  \leq \frac{C_\Lambda}{\delta^2}\|u^*\|_{L^\infty(B_\delta)}.
	\end{equation}
Note that when $f\equiv1$, the above estimate \eqref{C2at0} is classic for Laplace's equation, see \cite[Theorem 2.10]{GT}.
Consequently, from \eqref{combalpha} we obtain that
	\begin{equation} \label{C2u*}
		|D^2u^*(0)|  \leq C_1\frac{\|u^*\|_{L^\infty(B_\delta)}}{\delta^2} \left( \frac{\|u^*\|_{L^\infty(B_\delta)}}{\delta^2} +1 \right).  
	\end{equation}
\end{proof}

\begin{remark}\label{re23}
\emph{
It is worth noting that by \cite{W06}, it suffices to assume that $f$ is Dini continuous, namely $\int_0^1\frac{\omega(r)}{r}\,dr<\infty$, where $\omega(r)=\sup_{|x-y|<r}|f(x)-f(y)|$. In that case, $C_\Lambda$ in \eqref{C2at0} will depend on the Dini norm of $f$ instead of $|f|_\alpha$. 
}
\end{remark}

We can now derive an interior $C^2$ estimate for $u$ in terms of its modulus of convexity. 
\begin{corollary}\label{co21}
Under the same hypotheses of Theorem \ref{mainthm}, one has
	\begin{equation}\label{C2co}
		|D^2u(0)| \leq \frac{C_1R^2}{m_0^4}|Du|^6_{\partial B_R} + C_1,
	\end{equation}	
where $m_0=\inf_{\partial B_R}u$ given in Lemma \ref{lemball}.
\end{corollary}

\begin{proof}
By a rotation of coordinates we may assume that $u_{11}(0)=\inf_{e\in\mathbb{S}^1}u_{ee}(0)$. 
From \eqref{u2trans} and equation \eqref{MA}, we have
	\begin{equation}\label{C2uat0}
	\begin{split}
		|D^2u(0)| \leq \frac{C_0}{u_{11}(0)} & \leq C_0|D^2u^*(0)| \\
			& \leq C_1\frac{\|u^*\|_{L^\infty(B_\delta)}}{\delta^2} \left( \frac{\|u^*\|_{L^\infty(B_\delta)}}{\delta^2} +1 \right). 
	\end{split}
	\end{equation}
where the last step is due to \eqref{C2u*}. 
Note that from \eqref{pLeng},
	\begin{equation} \label{C0u*}
		\|u^*\|_{L^\infty(B_\delta)}  \leq |x_1||y_1|  \leq R\,|Du|_{\partial B_R}.
	\end{equation}
Observe that in \eqref{del} we may assume $|Du|_{\partial B_R}  > R$ (for otherwise, $|D^2u(0)|$ is bounded), and thus $\delta=\frac{m_0}{2|Du|_{\partial B_R}}$.
Combining $\delta$ and \eqref{C0u*} into \eqref{C2uat0}, we obtain
	\begin{equation}\label{C2uu}
	\begin{split}
		|D^2u(0)| &\leq C_1\frac{R\,|Du|_{\partial B_R}^3}{m_0^2}\left(\frac{R\,|Du|_{\partial B_R}^3}{m_0^2}+1\right) \\
			&\leq 2C_1\left(\frac{R^2\,|Du|_{\partial B_R}^6}{m_0^4}+1\right).
	\end{split}
	\end{equation}
\end{proof}

\begin{remark}
\emph{
The estimate \eqref{C2uu} reveals a relation between the interior $C^2$ bound and the strict convexity of $u$. 
Consider a typical solution of \eqref{MA} with $f\equiv1$,
	\begin{equation} \label{typical}
		u(x) = \frac{\eps}{2}x_1^2+\frac{1}{2\eps}x_2^2, 
	\end{equation}
one has $|D^2u(0)|=\eps^{-1}$, $|Du|_{\partial B_R}=\eps^{-1}R$ and $m_0=\eps R^2/2$.
The estimate \eqref{CHOest} obtained in \cite{CHO} is 
	\begin{equation*}\label{estChen}
		\eps^{-1} = |D^2u(0)| \leq C_1 e^{C_2 \sup|Du|^2/R^2} \leq  C_1e^{C_2\eps^{-2}},
	\end{equation*}
where the upper bound is of exponential order $e^{C\eps^{-2}}$. 
Instead, the estimate \eqref{C2uu} gives
	\begin{equation}
		\eps^{-1} = |D^2u(0)| \leq C_1\left( \eps^{-10} + 1 \right),
	\end{equation}
where the upper bound is of polynomial order $\eps^{-10}$.
}
\end{remark}

\subsection{Modulus of convexity}\label{s22}

In order to derive the estimate \eqref{est3} from \eqref{C2uu}, we need an estimate on the modulus of convexity $m$ of $u$, which is defined in \eqref{mconv}. 
Knowing $u$ is strictly convex in dimension two (or equivalently $m>0$) is not sufficient, it indeed requires an explicit estimate on the lower bound of $m$ in terms of $u$. 
By following the approach in \cite[\S3]{TW}, we are able to obtain that:

\begin{lemma}\label{lemc}
Let $m=m(t)$ be the modulus of convexity of $u$. 
For any $t>0$, assume that $m(t)$ is attained at $z=0$ and the support function $\ell_z\equiv0$. 
Denote $b:=|Du|_{\partial B_t}$. Then,
	\begin{equation}\label{sconv}
		m(t) \geq \frac{\frac12bt}{e^{C_2\frac{b^2}{t^2}} - 1},
	\end{equation}
where $C_2=128C_0$ is a universal constant.  
\end{lemma}

\begin{proof}
Assume that $m(t)=\inf_{\partial B_t}u$ is attained at $x=(t,0)$. 
On the negative $x_1$-axis, since $u(0)=0$ and $m(t)>0$, there is a constant $\tau\in(0,t]$ such that $u(-\tau,0)=m(t)$.

By convexity we have $0\leq u(x_1,0)\leq m(t)$ for $x_1\in(-\tau, t)$, and for any $x\in B_t\cap\{-\tau<x_1<t\}$,
	\begin{align*}
		u(x_1,x_2) & \geq u(x_1,0)-b|x_2|  \geq -b|x_2|, \\
		u(x_1,x_2) & \leq u(x_1,0)+b|x_2| \leq m(t) + b|x_2|.
	\end{align*}
It follows that for any $x\in(\frac14t,\frac34t)\times(-\frac14t,\frac14t)$,
	\begin{equation*}
		\partial_{x_1}u(x) \leq \frac{u(\frac78t,x_2) - u(x_1,x_2)}{\frac18t} \leq \frac8t\big(m(t)+2b|x_2|\big).
	\end{equation*}
Similarly we have $\partial_{x_1}u(x) \geq -\frac8t\big(m(t)+2b|x_2|\big)$. 

From equation \eqref{MA}, $u_{11}u_{22} \geq C_0^{-1}$. Hence
	\begin{equation*}
		\int_{\frac14t}^{\frac34t} \frac{1}{u_{22}}\,dx_1 \leq C_0 \int_{\frac14t}^{\frac34t} u_{11}\,dx_1 \leq \frac{16 C_0}{t}\big(m(t)+2b|x_2|\big).
	\end{equation*}
By H\"older's inequality, we obtain
	\begin{equation*}
		\int_{\frac14t}^{\frac34t} u_{22}\,dx_1 \geq \frac{t^2}{4} \left( \int_{\frac14t}^{\frac34t} \frac{1}{u_{22}}\,dx_1 \right)^{-1} \geq \frac{t^3}{64C_0} \big(m(t)+2b|x_2|\big)^{-1}.
	\end{equation*}	
It follows that
	\begin{align*}
		bt &\geq \int_{\frac14t}^{\frac34t} \left[ \int_0^{\frac14t} u_{22}\,dx_2 \right]\,dx_1 \\
		&=	\int_0^{\frac14t} \left[ \int_{\frac14t}^{\frac34t} u_{22}\,dx_1 \right]\,dx_2\\
		&\geq \frac{t^3}{64C_0} \int_0^{\frac14t} \big(m(t)+2b|x_2|\big)^{-1}\,dx_2.
	\end{align*}
Therefore,
	\begin{equation}
		\frac{64C_0b}{t^2} \geq \int_0^{\frac14t} \frac{dx_2}{m(t) + 2bx_2}.
	\end{equation}

After integration, one has
	\begin{equation}
		128C_0\frac{b^2}{t^2} \geq \log \frac{m(t) + bt/2}{m(t)},
	\end{equation}
and thus
	\begin{equation}
		e^{128C_0\frac{b^2}{t^2}} \geq 1+ \frac{bt/2}{m(t)},
	\end{equation}
which implies \eqref{sconv}, namely
	\begin{equation}
		m(t) \geq \frac{bt/2}{e^{128C_0\frac{b^2}{t^2}} - 1}.
	\end{equation}
\end{proof}

\begin{proof}[Proof of Theorem \ref{mainthm}] 
Now, we are ready to prove Theorem \ref{mainthm} by combining Corollary \ref{co21} and Lemma \ref{lemc}.
Denote $b=|Du|_{\partial B_R}$. From the assumption of Theorem \ref{mainthm}, one has $m_0=\inf_{\partial B_R}u\geq m(R)$. 
Hence, by \eqref{C2co} and \eqref{sconv} we have
	\begin{align}\label{C2inb}
		|D^2u(0)| &\leq C_1\left( \frac{R^2b^6}{m^4(R)} + 1 \right) \nonumber \\
			&\leq C_1\left( \frac{b^2}{R^2}\left(e^{C_2\frac{b^2}{R^2}}-1\right)^4 + 1 \right) \\
			&\leq C_1\left( \frac{b^2}{R^2}e^{C_2\frac{b^2}{R^2}} + 1 \right), \nonumber
	\end{align}
for some larger constants $C_1$ and $C_2$. 

By applying the above estimate \eqref{C2inb} in $B_{R/2}$ and noting that $\sup_{\partial B_{R/2}}|Du| \leq 4\sup_{B_R}|u|/R$, we can obtain
	\begin{equation}
		|D^2u(0)| \leq C_1 \left( \frac{\sup_{B_R}|u|^2}{R^{4}} e^{C_2  \frac{\sup_{B_R}|u|^2}{R^4}} + 1  \right), 
	\end{equation}
where the constant $C_1$ depends on $C_0$ and $\|f\|_{C^\alpha (B_R(0))}$, while $C_2$ depends only on $C_0$. 
\end{proof}

\subsection{Remarks}

\begin{itemize}
\item[(i)]
The above proof of Theorem \ref{mainthm} also applies to Monge-Amp\`ere equation \eqref{MA} with a general right hand side, that is
	\begin{equation*}
		\det\,D^2u(x) = f(x,u,Du)\quad	\text{ for } x\in B_R(0)\subset\mathbb{R}^2.
	\end{equation*}
If $f=f(x,z,p)$ is H\"older continuous with respect to all variables $x, z, p$, and satisfies $0<C_0^{-1}\leq f\leq C_0$, we can obtain the purely interior $C^2$ estimate \eqref{est3} by the same approach.  
\item[(ii)]
As mentioned in Remark \ref{re23}, by the estimate in \cite{W06}, we can reduce $f$ to be merely Dini continuous. 
\end{itemize}


\end{document}